\documentclass[11pt,a4paper]{article}
\usepackage{t1enc}
\usepackage[latin1]{inputenc}
\usepackage[english]{babel}
\usepackage{amsfonts}
\usepackage{amsthm}
\usepackage{amssymb}
\newtheorem{teo}{Theorem}[section]

\newtheorem{lem}[teo]{Lemma}

\newtheorem{defi}[teo]{Definition}
\newtheorem{rem}[teo]{Remark}

\def\p{{\cal P}}
\def\ga{G_{\cal A}}
\def\ah{{\cal A}_h}
\def\a{{\cal A}}

\begin{document}

\title{\vspace*{0cm}Unitarization of uniformly bounded subgroups in finite von Neumann algebras}

\date{}
\author{Mart\'in Miglioli\footnote{Supported by ANPCyT, Argentina.}}

\maketitle

%\begin{spacing}{.95}
\abstract{\footnotesize{\noindent }This note will present a new proof of the fact that every uniformly bounded group of invertible elements in a finite von Neumann algebra is similar to a unitary group. The proof involves metric geometric arguments in the non-positively curved space of positive invertible operators of the algebra; in 1974 Vasilescu and Zsido proved  this result using the Ryll-Nardzewsky fixed point theorem.
%\end{spacing}

\setlength{\parindent}{0cm} %% para que no indente los parrafos nuevos

\section{Geometry of the cone of positive operators in a finite algebra}

The metric geometry of the cone of positive invertible operators in a finite von Neumann algebra was studied in \cite{andruchowlarotonda,condelarotonda}. In this subsection we recall some facts from these papers.

Let $\a$ be a von Neumann algebra with a finite (normal, faithful) trace $\tau$. Denote by $\ah$ the set of selfadjoint elements of $\a$, by $\ga$ the group of invertible elements, by $U_{\a}$ the group of unitary operators, and by $\p$ the set of positive invertible operators

$$\p = e^{\ah} = \{ a \in \ga : a > 0 \};$$

$\p$ is an open subset of $\ah$ in the norm topology. Therefore if one regards it as a manifold, its tangent spaces identify with $\ah$ endowed with the uniform norm $\|\cdot\|$.

We make of $\p$ a weak Banach-Finsler manifold by assigning for each $a\in \p$ the following $2$-norm to the tangent space $T_a(\p)\simeq \ah$

$$\|x\|_{a,2}= \|a^{-\frac12}xa^{-\frac12}\|_2,\qquad  \mbox{for } x\in \ah \simeq T_a(\p)$$

where

$$\|x\|_{2}=\tau(x^2)^{\frac12} \qquad \mbox{for } x\in \ah .$$

One obtains a geodesic distance $d_2$ on $\p$ by considering

$$d_2(a,b)=inf\{Lenght(\gamma): \gamma \mbox{ is a piecewise smooth curve joining } a \mbox{ and } b\},$$

where smooth means differentiable in the norm induced topology and the lenght of a curve $\gamma : [0,1] \to \p$ is measured using the norm above:

$$Lenght(\gamma)=\int_0^1 \|\dot{\gamma}(t)\|_{\gamma(t),2}dt.$$

If $\a$ is finite dimensional, i.e. a sum of matrix spaces, this metric is well-known: it is the non positively curved Riemannian metric on the set of positive definite matrices \cite{mostow}.

If $\a$ is of type $II_1$, the trace inner product is not complete, so that $\p$ is not a Hilbert-Riemann manifold and $(\p,d_2)$ is not a complete metric space, see Remark $3.21$ in \cite{condelarotonda}.

The following holds

\begin{itemize}

\item By \cite[Th. 3.1 and Th. 3.2]{andruchowlarotonda} the unique geodesic between $a$ and $b$ for $a,b\in \p$ is given by

$$\gamma_{a,b}(t) = a^{\frac12}(a^{-\frac12}ba^{-\frac12})^ta^{\frac12}$$

and has lenght equal to

$$d_2(a,b):=Lenght(\gamma_{a,b}) = \|ln(a^{-\frac12}ba^{-\frac12})\|_2.$$

\item The action of $\ga$ on $\p$ given by $I_g(a)=gag^*$ is isometric, i.e. $d_2(I_g(a),I_g(b))=d_2(a,b)$, and sends geodesic segments to geodesic segments, i.e. $I_g \circ \gamma_{a,b} =\gamma_{I_g(a),I_g(b)}$ for all $a,b\in \p$ and $g \in \ga$. See the Introduction of \cite{andruchowlarotonda}.

\item Let $a \in \p$ and $\gamma:[0,1] \to \p$ be a geodesic. Then \cite[Theorem 4.4]{condelarotonda}

$$d_2(\gamma_0,\gamma_1)^2 + 4d_2(a,\gamma_{\frac12})^2 \leq 2(d_2(a,\gamma_0)^2 + d_2(a,\gamma_1)^2)$$

so the metric space $(\p,d_2)$ satisfies the semi-parallelogram law (see Definition \ref{semipar} below).

\item By \cite[Cor. 3.4]{andruchowlarotonda} the distance along two geodesics is convex, i.e. $t\mapsto d_2(\gamma_{a_1,b_1}(t),\gamma_{a_2,b_2}(t))$, $[0,1]\to [0,+\infty)$ is convex for $a_1,b_1,a_2,b_2\in \p$. This implies

$$d_2(\gamma_{a_1,b_1}(t),\gamma_{a_2,b_2}(t)) \leq td_2(\gamma_{a_1,b_1}(0),\gamma_{a_2,b_2}(0))+(1-t)d_2(\gamma_{a_1,b_1}(1),\gamma_{a_2,b_2}(1))=$$
$$=td_2(a_1,a_2)+(1-t)d_2(b_1,b_2).$$

If $t_0\in [0,1]$ is fixed, the continuity of $(a,b)\mapsto \gamma_{a,b}(t_0)$, $\p \times \p \to \p$ in the $d_2$ metric follows from the above inequality.

\item Let $\p_{c_1,c_2}:=\{a \in \p: c_11 \leq a \leq c_21\}$ for $0 < c_1 <c_2$. In $\p_{c_1,c_2}$ the linear metric and the rectifiable distance are equivalent \cite[Prop. 3.2]{condelarotonda}, i.e. there are $C>0$, $C' >0$ such that

$$\|a-b\|_2 \leq Cd_2(a,b),\qquad d_2(a,b) \leq C'\|a-b\|_2 \qquad a,b\in \p_{c_1,c_2}$$

Since $\|\cdot\|_2$ is complete on subsets of $\a$ which are closed and bounded in the uniform norm and $\p_{c_1,c_2}$ is closed and bounded in the uniform norm  $(\p_{c_1,c_2},d_2)$ is a complete metric space.

Also, for $a,b\in \p_{c_1,c_2}$

$$d_2(a,b) \leq C'\|a-b\|_2 \leq C'\|a-b\| \leq 2C'c_2$$

so that $\p_{c_1,c_2}$ is bounded in the $d_2$ metric.

\item $\p_{c_1,c_2}$ is geodesically convex: if $a,b\in \p_{c_1,c_2}$ then $\gamma_{a,b}(t) \in \p_{c_1,c_2}$ for every $t \in [0,1]$, see \cite{andcorstoj}.

\end{itemize}

\section{Non-negatively curved metric spaces}

In this subsection we recall some well-known results from metric geometry. A general reference is \cite{burago}. For the convenience of the the reader we include the proof of the Bruhat-Tits fixed point theorem.

\begin{defi}\label{semipar}

A metric space $(X,d)$ satisfies the semi-parallelogram law if for all $x,y \in X$ there is a $z \in X$ such that for all $w\in X$ the following inequality holds

$$d(x,y)^2 + 4d(w,z)^2 \leq 2[d(x,z)^2 +d(y,z)^2].$$

A Bruhat-Tits space is a complete metric space in which the semi-parallelogram law holds.

\end{defi}

\begin{rem}

The point $z$ satisfying this inequality is unique and is called the midpoint between $x$ and $y$ and we denote it by $m(x,y)$. We therefore have a function $m: X\times X \to X$ called the midpoint map.

\end{rem}

\begin{lem} \textsc{Serre's Lemma \cite[Ch. XI, Lemma 3.1]{lang}}

Let $(X,d)$ be a Bruhat-Tits and $S$ a bounded subset of $X$. Then there is a unique closed ball $B_r[y]$ of minimal radius containing $S$.

\end{lem}

\begin{defi}

The center $y$ of the closed ball $B_r[y]$ in the previous lemma is called the circumcenter of the bounded set $S$.

\end{defi}

\begin{teo} \textsc{Bruhat-Tits fixed point theorem \cite{bruhattits}}

If $(X,d)$ is a Bruhat-Tits space and $I:G\to Isom(X)$ is an action of a group $G$ on $X$ by isometries which has a bounded orbit, then the circumcenter of each orbit is a fixed point of the action.

\end{teo}

\begin{proof}

We denote the action by $g\cdot x$ for $g\in G$ and $x\in X$. Since the action is isometric and there is a bounded orbit all orbits are bounded. For $x\in X$ let $B_r[y]$ be the unique closed ball of minimal radius which contains $G\cdot x$. If $g\in G$ then $G\cdot x= g\cdot (G\cdot x)\subseteq g\cdot B_r[y]=B_r[g\cdot y]$ where the last equality follows since the action is isometric. From the uniqueness of the closed balls of minimal radius containing $G\cdot x$ we conclude that $g\cdot y=y$. Therefore, $g\cdot y=y$ for every $g\in G$ and $y$ is a fixed point of the action.

\end{proof}

\section{Uniformly bounded subgroups}

\begin{defi}

A subset $A \subseteq \p$ is geodesically convex if $\gamma_{a,b}(t) \in A$ for every $a,b \in A$ and $t\in [0,1]$.

\end{defi}

\begin{defi}\label{convhull}

The convex hull of a subset $S \subseteq \p$ is the smallest geodesically convex set containing $S$ and we denote it by $conv(A)$.

An alternative definition is

$$conv(S)= \bigcup_{n \in \mathbb{N}}X_n$$

where $X_1=S$, and $X_{n+1}=\{\gamma_{a,b}(t): a,b \in X_n, t \in [0,1]\}$ inductively for $n \geq 1$.

\end{defi}

\begin{lem}

If $C \subseteq \p_{c_1,c_2}$ is a geodesically convex subset then its closure $\overline{C}$ in $(\p_{c_1,c_2},d_2)$ is geodesically convex.

\end{lem}

\begin{proof}

If $a,b \in \overline{C}$ and $t \in [0,1]$ let $(a_n)_n,(b_n)_n$ be sequences in $C$ such that $a_n \to a$, $b_n\to b$. $\gamma_{a_n,b_n}(t)\in C$ for all $n\in \mathbb{N}$ and since $(a,b)\mapsto \gamma_{a,b}(t)$ is continuous on $\p_{c_1,c_2}\times \p_{c_1,c_2}$, $\gamma_{a_n,b_n}(t) \to \gamma_{a,b}(t)$. We conclude that $\gamma_{a,b}(t)\in \overline{C}$.

\end{proof}

\begin{teo}

Let $H \subseteq \ga$ be a uniformly bounded subgroup, i.e. $\sup_{h\in H}\|h\| := M < \infty$. Then there is an $s \in \p_{M^{-1},M}$ such that $shs^{-1}\in U_{\a}$ for every $h \in H$.

\end{teo}

\begin{proof}

Consider the action $I: H \to Isom(\p)$ given by $I_h(a)=hah^*$ for $h\in H$ and $a\in \p$. We denote $h\cdot a:=I_h(a)$. Since $H\cdot 1 = \{hh^*: h\in H\} \subseteq \p_{M^{-2},M^2}$ and $\p_{M^{-2},M^2}$ is geodesically convex $conv(H\cdot 1) \subseteq \p_{M^{-2},M^2}$. Also, since $\p_{M^{-2},M^2}$ is closed in $(\p,d_2)$, $\overline{conv}(H\cdot 1) \subseteq$ $\p_{M^{-2},M^2}$.

We adopt the notation of Definition \ref{convhull}. $X_1 = H \cdot 1$ is invariant for the action and since the action sends geodesics segments to geodesic segments, if $X_n$ is invariant then $X_{n+1}$ is invariant for all $n \geq 1$. We conclude that $conv(H\cdot 1)= \bigcup_{n \in \mathbb{N}}X_n$ is invariant. Since the action is also isometric $\overline{conv}(H\cdot 1)$ is an invariant subset and we can restrict the action to this subset.

Note that $\overline{conv}(H\cdot 1)$ is a geodesically convex subset of $\p$, in $(\p,d_2)$ the semi-parallelogram holds and the midpoint of $a,b\in \p$ is $\gamma_{a,b}(\frac12)$, so this law also holds in $(\overline{conv}(H\cdot 1),d_2)$. Since $\overline{conv}(H\cdot 1)$ is a closed subset of the complete metric space $(\p_{M^{-2},M^2},d_2)$, $(\overline{conv}(H\cdot 1),d_2)$ is a complete metric space. We conclude that $(\overline{conv}(H\cdot 1),d_2)$ is a Bruhat-Tits space.

Since $\p_{M^{-2},M^2}$ is bounded in the $d_2$ metric $\overline{conv}(H\cdot 1)$ is bounded in this metric. Therefore the action has bounded orbits and the Bruhat-Tits fixed point theorem states that the circumcenter $a\in \overline{conv}(H\cdot 1)$ of $H\cdot 1$ satisfies $I_h(a)=hah^*=a$ for all $h\in H$. Then

$$1=a^{-\frac12}aa^{-\frac12}=a^{-\frac12}hah^*a^{-\frac12}=(a^{-\frac12}ha^{\frac12})(a^{\frac12}h^*a^{-\frac12}) $$
$$=(a^{-\frac12}ha^{\frac12})(a^{-\frac12}ha^{\frac12})^* \qquad \textrm{for all $h\in H$}$$

so that $a^{-\frac12}Ha^{\frac12} \subseteq U_{\a}$.

Since  $a\in \p_{M^{-2},M^2}$, then $a^{\frac12} \in \p_{M^{-1},M}$ because the square root is an operator monotone function \cite[Prop. 4.2.8]{kadring}. Taking $s=a^{\frac12}$ we get the unitarizer stated in the theorem.

\end{proof}

\begin{rem}
The unitarizability of a uniformly bounded subgroup $H$ of the group of bounded linear operators acting on a Hilbert space was obtained independently in the 50s by Day, Dixmier and Nakamura, see \cite{nakamura} and the references therein, assuming that $H$ is amenable. In that context the unitarizer $s$ was obtained as the square root of the center of mass of $\{hh^*\}_{h\in H}$. In the present note the unitarizer is the square root of the circumcenter of that same set; we assume however the existence of a finite trace, and in this setting, Vasilescu and Zsido \cite{vasilescuzsido} proved in the 70s the result (without the assumption on amenability) using the Ryll-Nardzewsky fixed point theorem which involves weak topologies. 

\end{rem}

\bigskip

\noindent

\end{document}